\documentclass{amsart}
\let\Hold\H
\def\N{{\mathbb{N}}} 
\def\R{{\mathbb{R}}} 
\usepackage{amssymb}

\newtheorem{theorem}{Theorem}[section]
\newtheorem{lemma}[theorem]{Lemma}

\theoremstyle{definition}

\newtheorem{conjecture}[theorem]{Conjecture}

\theoremstyle{remark}

\numberwithin{equation}{section}
\usepackage{url,hyperref}

\DeclareMathOperator{\disc}{disc}
\DeclareMathOperator{\herdisc}{herdisc}

\newcommand{\junk}[1]{}

\begin{document}
\title[Hereditary Discrepancy of HAPs]{On The Hereditary Discrepancy of Homogeneous Arithmetic  Progressions}
\author{Aleksandar Nikolov \and Kunal Talwar}

\begin{abstract}
  We show that the hereditary discrepancy of homogeneous arithmetic  progressions is lower bounded by $n^{1/O(\log \log n)}$. This bound  is tight up to the constant in the exponent. Our lower bound goes  via proving an exponential lower bound on the discrepancy of set  systems of subcubes of the boolean cube $\{0, 1\}^d$.
\end{abstract}

\maketitle

\section{Introduction}
Circa 1932 Paul Erd{\Hold{o}}s made the following conjecture:
\begin{conjecture}[\cite{erdHos1957some}]\label{conj:EDP}
  For any function $f:\N \rightarrow \{-1, +1\}$ and for any constant
  $C$, there exist positive integers $n$ and $a$ such that
  \begin{equation*}
    |\sum_{i = 1}^{\lfloor n/a \rfloor}{f(ia)}| > C. 
  \end{equation*}
\end{conjecture}

This question can be phrased in the language of discrepancy theory as
follows. For a positive integer parameter $n$, we consider the set system of
subsets of $\{1, \ldots, n\}$ given by arithmetic progressions of the
form $(ia)_{i = 1}^{k}$ for all positive integers $a \leq n$ and $k
\leq \lfloor n/a \rfloor$. As is customary, we shall call such
arithmetic progressions \emph{homogeneous}. The \emph{discrepancy} of
a function  $f:\{1, \ldots, n\} \rightarrow \{-1, 1\}$ for this set
system is the maximum
value of $|\sum_{i = 1}^{k}{f(ia)}|$ over all $a$ and $k$ as above. The
discrepancy of the set system of homogeneous arithmetic progressions
over $\{1, \ldots, n\}$ is the minimum achievable discrepancy over all
functions $f:\{1, \ldots, n\} \rightarrow \{-1, 1\}$. In this
language, Conjecture~\ref{conj:EDP} states that the discrepancy of
homogeneous arithmetic progressions is unbounded as $n$ goes to
infinity.

This problem is now known as the Erd\Hold{o}s Discrepancy Problem, and
stands as a major open problem in discrepancy theory and combinatorial
number theory. Relying on a computer-aided proof, Konev and Lisitsa
recently reported~\cite{KonevL14} that the discrepancy of
homogeneous arithmetic progressions over $\{1, \ldots, n\}$ 
is at least $3$ for large enough $n$, and this
remains the best known lower bound (a lower bound
of $2$ for $n\geq 12$ was well known). On the other hand, the function
$f$ which takes value $f(i) = -1$ if and only if the last nonzero
digit of $i$ in ternary representation is $2$ has discrepancy $O(\log
n)$. For references and other partial results related to the
Erd\Hold{o}s Discrepancy Problem, see~\cite{finchrefs, polymath-wiki}.

The Erd\Hold{o}s Discrepancy Problem  recently also received attention as
the subject of the fifth polymath project~\cite{polymath-wiki}. Our
note is motivated by results of Alon and Kalai, announced and
sketched in the weblog post~\cite{kalai-post}. Using the Beck-Fiala
theorem, they showed that even for homogeneous arithmetic progressions
restricted to an \emph{arbitrary subset} of the integers up to $n$,
the discrepancy is no more than $n^{1/\Omega(\log \log n)}$. Also, for
infinitely many $n$, they constructed a set of integers $W_n$ all
bounded by $n$, so that there is a set of homogeneous arithmetic
progressions which, when restricted to $W_n$, form a known high
discrepancy set system (the Hadamard set system). This construction
showed that the minimum discrepancy for homogeneous arithmetic
progressions restricted to $W_n$ is at least $\Omega(\sqrt{\log
  n}/\sqrt{\log \log n})$. Since their discrepancy upper bound only
uses a bound on the number of distinct homogeneous arithmetic
progressions any integer less than $n$ belongs to, it was
reasonable to guess that the lower bound was closer to the truth.

In this note we show that in fact it is the upper bound of Alon and
Kalai which is tight up to the constant in the exponent. Our main
result is given by the following theorem.
\begin{theorem}\label{thm:mainlb}
  For infinitely many positive integers $n$, there exists a set $W_n
  \subseteq \{1, \ldots, n\}$ of square-free integers such that the
  following holds. For any $f:W_n \rightarrow \{-1, +1\}$ there exists a
  positive integer $a$ so that
  \begin{equation*}
    |\sum_{\substack{b\in W_n\\ a|b}}{f(b)}| = n^{1/O(\log \log n)}. 
  \end{equation*}
\end{theorem}
Our construction of the sets $W_n$ is inspired by the construction of
Alon and Kalai. Instead of the Hadamard set system, we embed a set
system of subcubes of the boolean cube inside the set of homogeneous
arithmetic progressions. Such systems of boolean subcubes were
previously considered in computer science in the context of private
data analysis~\cite{shiva2010,rudelson2011row}, and by Chazelle and
Lvov~\cite{ChazelleL01-subcubes} as a tool to prove a polynomial lower
bound on the discrepancy of axis-aligned boxes in high dimension.  We
give a new simpler proof of an improved lower bound on the discrepancy
of boolean subcubes, using elementary Fourier analysis and the
determinant lower bound on hereditary discrepancy due to Lov\'{a}sz,
Spencer, and Vesztergombi~\cite{lovasz1986discrepancy}.

Our construction produces sets $W_n$ of square free integers with a
large number of prime divisors, suggesting that such integers are a
chief obstacle in achieving bounded discrepancy for homogeneous
arithmetic progressions.

\section{Preliminaries} 
For a positive integer $n$, let $[n]$ be
the set $\{1, \ldots, n\}$. Given a set $S$, let ${S \choose k}$ be
the set of cardinality $k$ subsets of $S$. The expression $\langle
\cdot, \cdot \rangle_2$ denotes the standard inner product over the
vector space $\mathbb{F}_2^d$. We identify elements of
$\mathbb{F}_2^d$ with the boolean cube $\{0, 1\}^d$ in the natural
way. We use $|v|$ for the Hamming weight of a vector $v$, i.e. $|v| =
|\{i: v_i = 1\}|$.

A \emph{set system} is defined as a pair
$(\mathcal{S}, U)$, where $\mathcal{S} = \{S_1, \ldots, S_m\}$ and
$\forall j \in [m]: S_j \subseteq U$. The \emph{restriction}
$(\mathcal{S}|_W, W)$ of a set system $(\mathcal{S}, U)$ to some $W
\subseteq U$ is defined by $\mathcal{S}|_W = \{S_1 \cap W, \ldots,
S_m \cap W\}$.

The discrepancy and hereditary discrepancy of a
set system $\mathcal{S}$ are defined as 
\begin{align*}
  \disc(\mathcal{S}) &= \min_{f:U \rightarrow \{-1, +1\}} \max_{j
    \in [m]}{| \sum_{i \in S_j}{f(i)}|} \\
  \herdisc(\mathcal{S}) &= \max_{W \subseteq U} \disc(\mathcal{S}|_W)
\end{align*}

The definitions of discrepancy and hereditary discrepancy can be
extended to matrices $A \in \R^{m \times n}$ in a natural
way. Analogously to the definition of a restriction of a set system,
we define a restriction $A|_W$ of $A \in \R^{m \times n}$ for $W
\subseteq [n]$ as the submatrix of columns of $A$ indexed by elements
of $S$. Then discrepancy and hereditary discrepancy for matrices are
defined as
\begin{align*}
  \disc(A) &= \min_{x \in \{-1, +1 \}^n}{\|Ax\|_\infty}\\
  \herdisc(A) &= \max_{W \subseteq [n]}{\disc(A|_W)}
\end{align*}

We will need the determinant lower bound on hereditary discrepancy,
due to Lov\'{a}sz, Spencer, and Vesztergombi.

\begin{theorem}[\cite{lovasz1986discrepancy}]\label{thm:detlb}
  For any real $m \times n$ matrix ${A}$,
  \begin{equation*}
    \herdisc(A) \geq \frac{1}{2} \max_k \max_B |\det(B)|^{1/k}
  \end{equation*}
  where, for any $k$, $B$ ranges over all $k \times k$ submatrices of
  $A$. 
\end{theorem}

\section{Proof of the Main Theorem}

Theorem~\ref{thm:mainlb} is a consequence of a lower bound on the
hereditary discrepancy of the set system of subcubes of the boolean
cube. Next we define this set system formally. For a positive integer
$d$, we define the set system $(\mathcal{S}^d, \{0, 1\}^d)$, where
$\mathcal{S}^d = \{S_v\}_{v \in \{0, 1, *\}^d}$ is defined by
\begin{equation*}
  S_v = \{u \in \{0, 1\}^d: v_i \neq * \Rightarrow u_i = v_i\}. 
\end{equation*}

Similar set systems were studied in computer science in relation to
computing conjunction queries on a binary database under the
constraint of differential
privacy~\cite{shiva2010,rudelson2011row}. The system $\mathcal{S}^d$
was also considered by Chazelle and Lvov in their study of the
discrepancy of high-dimensional axis-aligned
boxes~\cite{ChazelleL01-subcubes}. They used the trace
bound~\cite{ChazelleL01-tracebound} to prove that
$\herdisc(\mathcal{S}^d) = \Omega(2^{cd})$ where $c$ is a constant
approximately equal to $c \approx 0.0477$. Here we slightly improve
the constant in the exponent, and give a simpler proof using elementary Fourier
analysis and the determinant lower bound.  

\begin{lemma}\label{lm:bcube-lb}
  For all positive integers $d$, $\herdisc(\mathcal{S}^d) =
  \Omega(2^{d/16})$. 
\end{lemma}

In the remainder of this section we prove that Lemma~\ref{lm:bcube-lb}
implies Theorem~\ref{thm:mainlb}. We prove Lemma~\ref{lm:bcube-lb} in
the subsequent section. 

\begin{proof}[Proof of Theorem~\ref{thm:mainlb}]
  For each positive integer $d$, we will construct a set of integers
  $B_d$ such that the hereditary discrepancy of homogeneous arithmetic
  progressions restricted to $B_d$ is lower bounded by the hereditary
  discrepancy of $\mathcal{S}^d$. Then Theorem~\ref{thm:mainlb} will
  follow from Lemma~\ref{lm:bcube-lb}.

  Let $p_{1,0} < p_{1,1} < \ldots < p_{d,0} < p_{d,1}$ be the first $2d$
  primes. We define $B_d$ to be the following set of square free integers 
  \begin{equation*}
    B_d = \{\prod_{i = 1}^d{p_{i, u_i}}: u \in \{0, 1\}^d\}.
  \end{equation*}
  In other words, $B_d$ is the set of all integers that are divisible by
  exactly one prime $p_{i, b}$ from each pair $(p_{i,0}, p_{i, 1})$ and
  no other primes. By the prime number theorem,  the largest of these
  primes satisfies $p_{d, 1} = \Theta(d
  \log d)$. Let $n = n(d)$ be the largest integer in $B_d$. The crude
  bound $n(d) = 2^{O(d \log d)}$ will suffice for our purposes. Notice
  that $d = \Omega(\log n/\log \log n)$.

  There is a natural one to one correspondence between the set $B_d$ and
  the set $\{0, 1\}^d$: to each $u \in \{0, 1\}^d$ we associate the
  integer $b_u = \prod_{i = 1}^d{p_{i, u_i}}$. By this correspondence,
  we can think of any assignment $f:\{0, 1\}^d \rightarrow \{-1, +1\}$
  as an assignment $f:B_d \rightarrow \{-1, +1\}$. We also claim that
  each set in the set system $\mathcal{S}^d$ corresponds to a homogeneous
  arithmetic progression restricted to $B_d$. With any $S_v \in
  \mathcal{S}^d$ (where $v \in \{0, 1, *\}^d$) associate the integer
  $a_v = \prod_{i: v_i \neq *}{p_{i, v_i}}$. Observe that for any $b_u
  \in B_d$, $a_v$ divides $b_u$ if and only if $u \in S_v$. We have the
  following implication for any assignment $f$, any $U \subseteq \{0,
  1\}^d$, and the corresponding $W = \{b_u: u\in U\}$:
  \begin{equation}
    \label{eq:bcube-hap}
    \exists S_v: |\sum_{u \in S_v \cap U}{f(u)}| \geq D \;\;\;
    \Leftrightarrow \;\;\;
    \exists a \in \N: |\sum_{\substack{b \in W\\ a | b}}{f(b)}| \geq D.
  \end{equation}
  Notice again that we treat $f$ as an assignment both to elements of
  $\{0, 1\}^d$ and to integers in $B_d$ by the correspondence $u
  \leftrightarrow b_u$. Lemma~\ref{lm:bcube-lb} guarantees the
  existence of some $U$ such that the left hand side of
  \eqref{eq:bcube-hap} is satisfied with $D = 2^{\Omega(d)} =
  n^{1/O(\log \log n)}$ for any $f$. Theorem~\ref{thm:mainlb} follows
  from the right hand side of \eqref{eq:bcube-hap}.
\end{proof}

\section{Lower Bounding the Discrepancy of $\mathcal{S}^d$}

It is convenient to first prove an easier lower bound on the
hereditary discrepancy of low-weight characters of
$\mathbb{F}_2^d$. Then we show that
an exponential (in $d$) lower bound on the discrepancy of characters
of weight $d/8$ implies an exponential lower bound on
$\mathcal{S}^d$. This approach is inspired by the noise lower bounds
on differential privacy  in~\cite{shiva2010}\junk{; however, discrepancy
presents distinct challenges, and the proofs in~\cite{shiva2010}
cannot be directly transported to our setting}.

As usual, for $v \in \mathbb{F}^d$ we define the character $\chi_v$ by
\begin{equation*}
  \forall u \in \{0, 1\}^d: \chi_v(u) = (-1)^{\langle v, u \rangle_2}. 
\end{equation*}
We refer to the Hamming weight $|v|$ of $v$ (taken as a binary vector)
as the \emph{weight} of the character $\chi_v$.  The matrix of the
Walsh-Hadamard transform is defined as $H_d = (\chi_v)_{v \in \{0,
  1\}^d}$, where each $\chi_v$ is written as a row vector of dimension
$2^d$. Notice that for any $v \neq w$, $\sum_{u \in \{0,
  1\}^d}{\chi_v(u)\chi_w(u)} = 0$, i.e.~$H_d$ is an orthogonal matrix;
each row of $H_d$ has squared Euclidean norm $\sum_{u \in \{0,
  1\}^d}{\chi_v(u)^2} = 2^d$.

We will be interested in a submatrix of $H_d$. For the remainder of
this note we assume that $d$ is divisible by 8; this is purely for
notational convenience: our arguments can easily be adapted to the
case when $d$ is not divisible by 8. Let $G_d = (\chi_v)_{v: |v| =
  d/8}$.  Notice that $G_dG_d^T = 2^d
I_{M}$ where $M= {d \choose d/8}$ and $I_M$ is the $M$-dimensional
identity matrix. Therefore,
\begin{equation}
  \label{eq:det-G}
  \det(G_dG_d^T) = (2^d)^{{d \choose d/8}}
\end{equation}

Given \eqref{eq:det-G} and using the determinant lower bound, we can derive a lower bound on the hereditary
discrepancy of $G_d$.

\begin{lemma} \label{lm:G-lb}
  For positive integers $d$,
  \begin{equation*}
    \herdisc(G_d) \geq \frac{2^{3d/16}}{2e}
  \end{equation*}
\end{lemma}
\begin{proof}
  Let $N= 2^d$ and let $M = {d \choose d/8}$. By \eqref{eq:det-G} and the
  Binet-Cauchy formula for the determinant, we have
  \begin{equation*}
    N^{M} = \det(G_dG_d^T) = \sum_{W \in {[N] \choose M}}{\det(G_d|_W)^2}
  \end{equation*}
  By averaging, there exists a set $W \in {[N] \choose M}$ so that 
  \begin{equation}
    \label{eq:det-G-restr}
    |\det(G_d|_W)|^{1/M} \geq \sqrt{N}{N \choose M}^{-1/2M} \geq
    \sqrt{\frac{M}{e}} 
  \end{equation}
  For the second inequality above we used the bound ${N \choose M}
  \leq (Ne/M)^M$. Plugging in the lower bound $M = {d\choose d/8} \geq 
  2^{3d/8}$ in \eqref{eq:det-G-restr}, we have $|\det(G_d|_W)|^{1/M}
  \geq 2^{3d/16}e^{-1}$. The proof is completed by an application of
  Theorem~\ref{thm:detlb}.
\end{proof}

We are now ready to prove Lemma~\ref{lm:bcube-lb} by exhibiting a
connection between the discrepancy of $G_d$ and the discrepancy of
$\mathcal{S}^d$.

\begin{proof}[Proof of Lemma~\ref{lm:bcube-lb}]
  By Lemma~\ref{lm:G-lb}, it is enough to prove the following inequality:
  \begin{equation}
    \label{eq:G-vs-bcube}
    \herdisc(G_d) \leq 2^{d/8}\herdisc(\mathcal{S}^d)
  \end{equation}
  The key observation is that when $|v| = d/8$, we can express the
  character $\chi_v$ as a linear combination of the indicator
  functions of $2^{d/8}$ sets in $\mathcal{S}^d$. Moreover, the
  coefficients of the linear combination are $\pm 1$. Next we make
  this observation precise.

  Let $v$ be an arbitrary fixed element of $\mathbb{F}^d$ such that $|v|=
  d/8$, and let $1 \leq i_1 < i_2 < \ldots < i_{d/8} \leq d$ denote
  the coordinates $i$ such that $v_i = 1$. Given $w \in
  \{0, 1\}^{d/8}$, let its \emph{extension} $v(w) = \{0, 1, *\}^d$ be
  defined by 
  \begin{equation*}
    v(w)_i = 
    \begin{cases}
      w_\ell & \text{ if } i = i_\ell \text{ for some } \ell \in [d/8];\\
      * & \text{otherwise.}
    \end{cases}
  \end{equation*}
  
  We use the notation $\mathbf{1}_{v(w)}$ for the indicator function
  of the set $S_{v(w)}$. Let $r_{v(w)}$ be a representative from the
  set $S_{v(w)}$, say one obtained by replacing every $*$ in $v(w)$
  with $0$.  Taking $r_{v(w)}$ and the elements of $S_{v(w)}$ as
  elements of $\mathbb{F}_2^d$ in the standard way, for each $z \in
  S_{v(w)}$, $\langle v, z\rangle_2 = \langle v, r_{v(w)}\rangle_2$,
  since only the coordinates $z_{i_1}, \ldots, z_{i_{d/8}}$ affect the
  inner product. Thus, we can express $\chi_v(u)$ as the linear
  combination of these indicator functions:
  \begin{equation}
    \label{eq:char-bcube}
    \forall u \in \{0, 1\}^d: \chi_v(u) = \sum_{w \in \{0, 1\}^{d/8}}{(-1)^{\langle v, r_{v(w)}\rangle_2}\mathbf{1}_{v(w)}(u)}.
  \end{equation}

  For any set $U \subseteq \{0, 1\}^d$ and any $f: U\rightarrow \{-1,
  1\}$, we use \eqref{eq:char-bcube} to write the linear
  transformation $(G_d|_U)f$ in terms of discrepancy values of sets in
  $\mathcal{S}^d$ restricted to the set $U$:
  \begin{align}
    \sum_{u \in U}{\chi_v(u)f(u)} &= \sum_{u \in U}{\left(\sum_{w \in
          \{0, 1\}^{d/8}}{(-1)^{\langle v, r_{v(w)}
            \rangle_2}\mathbf{1}_{v(w)}(u)} \right)f(u)}\notag\\
    &= \sum_{w \in \{0, 1\}^{d/8}}{(-1)^{\langle v, r_{v(w)}  \rangle_2}\left(\sum_{u \in S_{v(w)} \cap U}{f(u)}\right)}. \label{eq:disc-char-bcube}
  \end{align}
  Let $f$ be the function that achieves
  $\disc(\mathcal{S}^d|_U)$. Each of the $2^{d/8}$ terms on the right
  hand side of \eqref{eq:disc-char-bcube} is then bounded in absolute
  value by $\disc(\mathcal{S}^d|_U)\leq
  \herdisc(\mathcal{S}^d)$. Since the choice of $U$ and $v$ was
  arbitrary, this proves \eqref{eq:G-vs-bcube}, and the lemma follows.\end{proof}

\section{Conclusion}

We presented a tight (up to the constant in the exponent) lower bound
on the hereditary discrepancy of homogeneous arithmetic
progressions. Our lower bound instances are given by a set of integers
in  $\{1,\ldots, n\}$ with a large number $\Theta(\log n/{\log \log
  n})$ of distinct prime factors. This suggests that integers with
many distinct factors are the main obstacle to achieving bounded
discrepancy for homogeneous arithmetic progressions.

Our discrepancy lower bound follows from a lower bound on the
discrepancy of a set system of subcubes of the boolean cube. Such set
systems have applications in the theory of differential privacy. The
ideas used in the proof of Lemma~\ref{lm:bcube-lb}, together with the
connection between discrepancy and differential privacy formalized
in~\cite{MNstoc} can be used to give simpler proofs of noise lower
bounds of the type considered in~\cite{shiva2010}. It is an
interesting question whether discrepancy bounds on set systems of
boolean subcubes can find other applications in combinatorics and
computer science. We leave open the question of characterizing the
exact discrepancy of such set systems.

\bibliographystyle{amsplain}
\bibliography{privacygeom}
\end{document}